\documentclass[a4paper,12pt]{amsart}
\usepackage{amssymb}
\usepackage{mathrsfs}
\usepackage{color}
\usepackage{epsfig}
\usepackage{graphicx}
\usepackage[all]{xy}
\theoremstyle{plain}
\newtheorem{theorem}{Theorem}[section]
\newtheorem{prop}[theorem]{Proposition}

\newtheorem{lemma}[theorem]{Lemma}

\theoremstyle{definition}
\newtheorem{remark}[theorem]{Remark}

\newtheorem{definition}[theorem]{Definition}
\newtheorem{example}[theorem]{Example}

\newtheorem{assumption}[theorem]{Assumption}

\newcommand{\C}{\mathbb{C}}
\newcommand{\R}{\mathbb{R}}
\newcommand{\Z}{\mathbb{Z}}

\newcommand{\SF}{\mathscr F}

\newcommand{\ind}{{\rm ind}}
\newcommand{\bD}{\mathbf D}

\renewcommand{\tilde}{\widetilde}
\renewcommand{\setminus}{\smallsetminus}
\newcommand{\nin}{/\kern-2.1ex\in}

\def\<{\left\langle}
\def\>{\right\rangle}
\def\End{\operatorname{End}}

\def\ind{\operatorname{ind}}

\def\supp{\operatorname{supp}}

\numberwithin{equation}{section}

\title[Cobordism invariance and the well-definedness of local index]{Cobordism invariance and \\ the well-definedness of local index}
\date{}
\author{Hajime Fujita}

\thanks{$^1$Partly supported by Grant-in-Aid for Young Scientists (B) 23740059 
and Young Scientists (B) 26800045.}

\address{Department of Mathematical and Physical Sciences
Japan Women's University
2-8-1 Mejirodai, Bunkyo-ku Tokyo, 112-8681 Japan}
\email{fujitah@fc.jwu.ac.jp}

\begin{document}

\maketitle

\begin{abstract}
In the previous papers,  Furuta, Yoshida and the author 
gave a definition of analytic index theory of Dirac-type 
operator on open manifolds   
by making use of some geometric structure on an open covering of the end of the open manifold and a perturbation of the Dirac-type operator. 
In this paper we show the cobordism invariance of the index, 
and as an application we show 
the well-definedness of the index with respect to the choice of the open covering. 
\end{abstract}

\tableofcontents

\section{Introduction}
In the series of papers \cite{Fujita-Furuta-Yoshida1, Fujita-Furuta-Yoshida2}, 
Furuta, Yoshida and the author developed an index theory of Dirac-type operators on open Riemannian manifolds when an additional structure, which is called an {\it acyclic compatible system}, is given on the end of the manifold. 
The resulting index, {\it local index} (or {\it relative index}), has several topological properties. 
In particular the {\it excision formula} leads us to have a localization 
formula of index of Dirac-type operators. 
Such a localization formula enables us to have geometric proofs of 
the localization formula of Riemann-Roch numbers 
to Bohr-Sommerfeld fibers and singular fibers of Lagrangian fibrations 
\cite{Fujita-Furuta-Yoshida1}, 
Danilov type formula for 4-dimensional locally toric Lagrangian fibrations 
\cite{Fujita-Furuta-Yoshida2}
and [Q,R]=0 formula for Hamiltonian torus actions \cite{Fujita-Furuta-Yoshida3}. 
To describe the index theory one needs to fix an open covering of the end;  
however, there are several choices of open coverings when one would like to apply the theory to individual manifold, for instance, a symplectic manifold with Hamiltonian torus action. 
It is not clear by definition that the local index does not depend on the choice of the open covering. 

In this paper we show that the well-definedness of the local index, that is, 
if there are two acyclic compatible systems with a specific condition, 
which we call {\it $G$-tangential},  on two ends of a single manifold 
such that the union is also an acyclic compatible system, then 
two local indices defined by two compatible systems coincide. 
The well-definedness is shown by the {\it cobordism invariance} of the local index. 
There are several works (\cite{Higsoncovinv}, \cite{Nicocovinv} and so on) concerning the cobordism invariance of the analytic index of Fredholm operators. 
The proof of the present paper is a variant of the technique developed by 
Braverman~\cite{Braverman, Bravermancobinv}.  
The technique is a version of Witten-type deformation by using a kind of Morita equivalence and an extra symmetry of Clifford algebras of indefinite metrics. 
Using the Morita equivalence the formulation of the cobordism invariance 
can be rephrased as a statement for Clifford module bundles with extra symmetry 
of Clifford algebras of indefinite metrics. 
Such a reformulation leads us to have an extension of the cobordism invariance for 
{\it real} Clifford module bundles. 
%In fact the real case is more fundamental than the complex case in the sense th%at the complex case can be shown by using the real case. 
The cobordism invariance for the complex 
Clifford module bundle can be proved by the almost same argument as in 
\cite{Braverman, Bravermancobinv}. 
We give the proof for the real Clifford module bundle case.

This paper is organized as follows. 
In Section~\ref{Clifford algebras for indefinite metrics} 
we recall definitions of Clifford algebras for a vector spaces/bundles with indefinite metrics and the modules over such algebras. 
In Section~\ref{Acyclic compatible systems and their local indices} 
we recall definitions of compatible fibrations, compatible systems and the 
definition of the local index. 
%Though almost all contents of this section are same as in 
%\cite[Section~2,3,4]{Fujita-Furuta-Yoshida3} 
We introduce a new class of compatible fibrations (resp. systems), 
{\it $G$-tangential} compatible fibrations (resp. systems), which is useful to define the local index and 
describe an application of the well-definedness of the local index. 
For latter convenience we work with Clifford module bundles with 
an action of Clifford algebras of Euclidean space with indefinite metric. 
In Subsection~\ref{Cobordism of compatible fibrations} and Subsection~\ref{Cobordism of compatible systems} we define the notion of cobordism 
of compatible fibrations and compatible systems. 
In Subsection~\ref{Cobordism invariance of local index} we give the statement of the cobordism invariance of the local index (Theorem~\ref{maintheorem}). 
%Theorem~\ref{maintheorem} can be shown by using a technique developed in \cite{}. 
In Section~\ref{Cobordism invariance via extra symmetry of Clifford algebras} 
we give a reformulation of the cobordism invariance using extra symmetry (Theorem~\ref{cobinvref}, Theorem~\ref{cobinvreal}). 
%Such a reformulation enable us to have a formulation of the cobordism invariance for real Clifford module bundles. 
In Section~\ref{Proof of Theorem{realcobinv}} we give the proof 
of Theorem~\ref{cobinvreal} by using a modification of the technique 
in \cite{Braverman, Bravermancobinv} for real Clifford module bundles. 
In Subsection~\ref{Comments on the complex case} we give comments on 
how we modify the proofs in \cite{Braverman, Bravermancobinv} or that for the real case 
to prove the complex case. 
In Section~\ref{Well-definedness of the local index} we show the well-definedness of the local index (Theorem~\ref{well def of local ind}). 
Finally we give an application of Theorem~\ref{well def of local ind}  
to the product of compatible fibration, which is convenient to compute the 
local index.

%%%%%%%%%%%%%%%%%%%%%%%%%%%%%%%%%%%%%%%%%%%%%%%%%%%%%%
\section{Clifford algebras for indefinite metrics}
\label{Clifford algebras for indefinite metrics} 
In this paper we will work with Clifford algebras for 
indefinite metrics and Clifford modules over such algebras. 
Let us recall basic definitions. 

\begin{definition}
Let $\R^{p,q}$ be the real vector space $\R^{p+q}$ 
with an indefinite metric $\cdot$, 
defined by 
$$
u\cdot v=(u_1v_1+\cdots+u_pv_p)-(u_1'v_1'+\cdots+u_q'v_q'
)
$$ 
for $u=(u_1, \cdots , u_p, u'_1,\cdots , u_q')$,  
$v=(v_1, \cdots , v_p, v_1', \cdots,  v_q')\in\R^{p,q}$. 
We define the {\it Clifford algebra} $Cl(\R^{p,q})$ over $\R^{p,q}$ by 
the algebra generated by $\R^{p,q}$ over $\R$ with the {\it Clifford relation} 
$$
uv+vu=-2u\cdot v 
$$ for $u,v\in \R^{p,q}$,  
which has the universal property among the above relation. 
\end{definition}

We often denote $Cl_{p,q}=Cl(\R^{p,q})$ for simplicity. 
By definition we have $Cl_{p,0}=Cl(\R^p)$. 

\begin{definition}
A {\it (complex) representation of $Cl(\R^{p,q})$} is a 
pair $(R, c_R)$ consisting of a Hermitian vector space $R$ and an $\R$-algebra homomorphism, the {\it Clifford action} (or {\it Clifford multiplication}), $c_R:Cl(\R^{p,q})\to {\rm End}_{\C}(R)$ such that $c_R(u)$ is a skew Hermitian 
operator for each $u\in \R^{p,0}\subset \R^{p,q}$ and 
$c_R(v)$ is a Hermitian operator 
for each $v\in \R^{0,q}\subset \R^{p,q}$. 
When $R$ has a $\Z/2$-grading and $c_R(u)$ is degree-one map for each $u\in \R^{p,q}$, 
we call $(R,c_R)$ a {\it $\Z/2$-graded} representation. 

By using a Euclidean vector space and (skew) symmetric operators, 
a {\it real representation} of $Cl(\R^{p,q})$ can be defined in  a 
similar manner. 
\end{definition}

\begin{example}
Let $W=W^+\oplus W^-$ be a $\Z/2$-graded representation of $Cl(\R^{p})$ and 
$\epsilon_W:W\to W$ be the map which represents the grading, i.e., 
$\epsilon|_{W^{\pm}}=\pm {\rm id}_W$. 
By forgetting the grading $W$ can be regarded as a representation of 
$Cl(\R^p\oplus\R\epsilon_W)$, where $\R\epsilon_W$ is the real vector space 
generated by $\epsilon_W$ with the negative definite metric $\epsilon_W\cdot\epsilon_W=-1$. 
\end{example}

%In this paper we mainly work with Hermitian representations unless mentioned. 
It is straightforward to define the Clifford algebra bundle $Cl(E)$ for a real vector bundle $E$ with metric. 
A ($\Z/2$-graded) complex (or real) $Cl(E)$-module bundle can be also defined. 
By definition we have $Cl(E_1\oplus E_2)\cong Cl(E_1)\otimes Cl(E_2)$ 
for two vector bundles $E_1$ and $E_2$. 

\begin{definition}
Let $(M,g)$ be a Riemannian manifold and 
$W$ a $Cl(TM\oplus\R^{p,q})$-module bundle over $M$. 
A differential operator $D:\Gamma(W)\to \Gamma(W)$ is called 
{\it Dirac-type} if $D$ is a formally self-adjoint operator of order-one 
whose principal symbol $\sigma(D):TM\to {\rm End}(W)$ satisfies the 
Clifford relation 
$\sigma(D)(v)\sigma(D)(v')+\sigma(D)(v')\sigma(D)(v)=
-2g(v,v')$ for any $v,v'\in TM$. 
In other words $\sigma(D)$ is equal to the Clifford action of 
$Cl(TM)\subset Cl(TM\oplus \R^{p,q})$. 
If $W$ is $\Z/2$-graded, then we assume that $D$ is of degree-one. 
\end{definition}

%%%%%%%%%%%%%%%%%%%%%%%%%%%%%%%%%%%%%%%%%%%%%%%%%%%%%%%%%%%%%%%%%%%%%%%%%%%%%%
\section{Acyclic compatible systems and their local indices}
\label{Acyclic compatible systems and their local indices} 
In this section we recall some definitions 
of compatible fibration, acyclic compatible system and 
their local indices following \cite{Fujita-Furuta-Yoshida2, Fujita-Furuta-Yoshida3}. 
\subsection{Compatible fibration}
Let $V$ be a manifold. 
%{\color{red}
\begin{definition}\label{compatible fibration}
A {\it compatible fibration on $V$} is a collection of the data 
$\{V_{\alpha}, {\SF}_{\alpha}\}_{\alpha\in A}$ consisting of 
an open covering $\{V_{\alpha}\}_{\alpha\in A}$ of $V$ and 
a foliation $\SF_{\alpha}$ on $V_{\alpha}$ with compact leaves 
which  satisfies the following properties.
\begin{enumerate}
\item The holonomy group of each leaf of $\SF_\alpha$ is finite. 
\item\label{correspondence between foliation and pi_alpha}For each $\alpha$ and $\beta$, if a leaf $L\in \SF_\alpha$ has non-empty intersection $L\cap V_\beta\neq \emptyset$, then, $L\subset V_\beta$. 
%\item\label{foliation on overlap}For each $\alpha$ and $\beta$, the set 
%\[
%\SF_{\alpha\beta}=\{ L_\alpha\cap L_\beta \mid L_\alpha\in \SF_\alpha ,\ L_\beta \in \SF_\beta \} 
%\]
%of the intersections of leaves of $\SF_\alpha$ and $\SF_\beta$ is a foliation 
%on $V_{\alpha}\cap V_{\beta}$. 
%%\item {\color{red}The holonomy group of each leaf of $\SF_{\alpha\beta}$ is finite.} 
\end{enumerate}
\end{definition}
%}

%Let $\{ V_\alpha ,\SF_\alpha\}_{\alpha \in A}$ be a compatible fibration on $V$. 

%%%%%%%%%%%%%%%%%%%%%%%%%%%%%%%
\subsection{Compatible system and its acyclicity}\label{Compatible system and its acyclicity}
Let $(V,g)$ be a Riemannian manifold, $W$ a $Cl(TV\oplus \R^{p,q})$-module bundle over $V$. 
% and $V$ an open subset of $M$ equipped with a compatible fibration $\{V_{\alpha}, {\SF}_{\alpha}\}_{\alpha\in A}$. 
In the rest of this paper we impose the following conditions on the Riemannian metric $g$. 
\begin{assumption}\label{assumption for Riemannian metric}
%\begin{itemize}
%\item 
Let $\nu_\alpha=\{ u\in TV_\alpha \mid g(u,v)=0\ \text{for all }v\in T\SF_\alpha \}$ be the normal bundle of $\SF_\alpha$. Then, $g|_{\nu_\alpha}$ is invariant under holonomy, and gives a transverse invariant metric on $\nu_\alpha$.
%\item On each $V_\alpha\cap V_\beta\neq \emptyset$, for $i=\alpha$, $\beta$, let $\nu_{\alpha\beta}^i$ be the normal bundle of $\SF_{\alpha\beta}$ in $T\SF_i|_{V_\alpha\cap V_\beta}$ which is defined by 
%\[
%\nu_{\alpha\beta}^i=\{ u\in T\SF_i|_{V_\alpha\cap V_\beta}\mid g(u,v)=0\ \text{for all }v\in T\SF_{\alpha\beta} \}
%\]
%Then, $g|_{\nu_{\alpha\beta}^i}$ gives a transverse invariant metric on $\nu_{\alpha\beta}^i$. 
%\item $\nu_{\alpha\beta}^\alpha$ and $\nu_{\alpha\beta}^\beta$ are perpendicular to each other with respect to $g|_{V_\alpha\cap V_\beta}$. 
%\end{itemize}
\end{assumption}
%\begin{remark}
%The above Riemannian metric is an orbifold version of a compatible Riemannian metric which is actually used in \cite{Fujita-Furuta-Yoshida2}. 
%\end{remark}

\begin{definition}\label{compatible system}
A {\it compatible system} on $(\{V_{\alpha}, \SF_{\alpha}\}, W)$ is a data $\{ D_{\alpha}\}_{\alpha \in A}$ satisfying the following properties. 
\begin{enumerate}
\item $D_{\alpha}\colon \Gamma (W|_{V_{\alpha}})\to \Gamma (W|_{V_{\alpha}})$ is an order-one formally self-adjoint differential operator.
\item $D_{\alpha}$ contains only the derivatives along leaves of $\SF_{\alpha}$. 
\item $D_{\alpha}$ is a Dirac-type operator along leaves. 
Namely 
the principal symbol of $D_{\alpha}$ is given by the composition of 
the dual of the natural inclusion $\iota_{\alpha}\colon T\SF_{\alpha}\to TV_{\alpha}$ and the Clifford multiplication 
$c\colon T^*\SF_{\alpha}\cong T\SF_{\alpha}\subset TV_{\alpha} \to \End (W|_{V_{\alpha}})$ . 
If $W$ is $\Z/2$-graded, then $D_{\alpha}$ is of degree-one. 
%$\sigma (D_{\alpha})=c\circ p_{\alpha}\circ \iota_{\alpha}^*\colon T^*V_{\alpha}\to \End (W|_{V_{\alpha}})$, where  is  from the tangent bundle along leaves of $\SF_\alpha$ to $TV_\alpha$, $p_{\alpha}\colon T^*\SF_{\alpha}\to T\SF_{\alpha}$ is the isomorphism induced by the Riemannian metric and is the  
\item For a leaf $L\in \SF_\alpha$ let $\tilde u\in \Gamma (\nu_\alpha|_L)$ be a section of $\nu_\alpha|_L$ parallel along $L$. 
%For $b\in U_{\alpha}$ and $u\in T_bU_{\alpha}$, let $\tilde{u}\in \Gamma(TV_{\alpha}|_{\pi^{-1}_{\alpha}(b)})$ be the horizontal lift of $u$ with respect to the Riemannian metric $g|_{V_\alpha}$. 
$\tilde{u}$ acts on $W|_L$ by the Clifford multiplication $c(\tilde{u})$. Then $D_{\alpha}$ and $c(\tilde{u})$ anti-commute each other, i.e. 
\[
0=\{ D_{\alpha},c(\tilde{u}) \}:=
D_{\alpha}\circ c(\tilde{u})+c(\tilde{u})\circ D_{\alpha}
\]
as an operator on $W|_L$. 
%\item If $V_{\alpha}\cap V_{\beta}\neq \emptyset$, then the anti-commutator $\{D_{\alpha},D_{\beta}\}:=D_{\alpha}\circ D_{\beta}+D_{\beta}\circ D_{\alpha}$ is a differential operator along leaves of $\SF_{\alpha\beta}$ of order at most two. 
\item $D_{\alpha}$ anti-commutes with $Cl_{p,q} (\subset Cl(TM\oplus \R^{p,q}))$-action. 
\end{enumerate}
\end{definition}

As in \cite[Lemma~3.4]{Fujita-Furuta-Yoshida2} for each leaf $L\in\SF_{\alpha}$ 
we have a small 
open tubular neighborhood $V_L$ of $L$ and the 
finite covering  $q_L:\tilde V\to V_L$ such that the induced foliation on 
$\tilde V_L$ is a bundle foliation with the projection $\pi_L:\tilde V_L\to \tilde U_L$. 

%We will use the notations $V_L$, $\tilde V_L$, and $q_L$ in the following definition. We also denote the projection map of the fiber bundle structure by $\pi_L\colon \tilde V_L\to\tilde U_L$. 

\begin{definition}\label{strongly acyclic}
A compatible system $\{ D_\alpha\}_{\alpha \in A}$ on $(\{V_{\alpha}, \SF_{\alpha}\}, W)$ is said to be {\it acyclic} if %there exists a collection of data $\{q_{L_\alpha}:\tilde V_{L_\alpha}\to V_{L_\alpha}, \ \pi_{L_\alpha}:\tilde V_{L_\alpha}\to \tilde U_{L_\alpha} \ | \ \alpha\in A, L_{\alpha}\in \SF_{\alpha}\}$ 
it satisfies the following conditions.  
\begin{enumerate}
\item The Dirac type operator 
$q^*_{L}D_{\alpha}|_{\pi_L^{-1}(\tilde b)}$ has zero kernel 
for each $\alpha\in A$, leaf $L\in \SF_{\alpha}$ and 
$\tilde b\in \tilde U_{L}$. 
\item If $V_\alpha\cap V_\beta\neq \emptyset$, then the anti-commutator 
$\{D_{\alpha},D_{\beta}\}:=D_{\alpha}D_{\beta}+D_{\beta}D_{\alpha}$ is a non-negative operator on $V_\alpha\cap V_\beta$. %$W|_{L_{\alpha\beta}}$ for each $L_{\alpha\beta}\in \SF_{\alpha\beta}$. 
\end{enumerate}
\end{definition}

%%%%%%%%%%%%%%%%%%%%%%%%%%%%%%%%%%
\subsection{Definition of the local index}\label{ind(M,V,W)}
In \cite{Fujita-Furuta-Yoshida2} Furuta, Yoshida and the author showed that 
when a manifold whose end is equipped with an 
acyclic compatible system with some technical conditions, 
the {\it local index} can be defined. 
In this section we first introduce a class of compatible fibration 
which gives a sufficient condition to define the {local index}.

\begin{definition}\label{tangential}
Suppose that a compact Lie group $G$ 
acts on a Riemannian manifold $V$ in an isometric way. 
Let $\{V_{\alpha}, \SF_{\alpha}\}_{\alpha\in A}$ be a compatible fibration on $V$. 
If the following conditions are satisfied, then we call the compatible fibration a {\it $G$-tangential compatible fibration} (or {\it tangential compatible fibration} for short). 
\begin{itemize}
\item $\{V_{\alpha}\}_{\alpha\in A}$ is a $G$-invariant open covering of $V$. 
\item Each leaf $L$ of $\SF_{\alpha}$ has 
positive dimension for all $\alpha\in A$. 
\item For each leaf $L$ of $\SF_{\alpha}$ there exists some $x\in V_{\alpha}$ 
such that $L$ is contained in the $G$-orbit $G\cdot x$. 
\end{itemize}
A compatible system on a $G$-tangential compatible fibration is 
called {\it $G$-tangential compatible system} 
(or {\it tangential compatible system} for short), 
\end{definition}

\begin{example}
In \cite[Definition~6.7]{Fujita-Furuta-Yoshida2} a 
class of compatible fibration which is called {\it good compatible fibration} is defined. 
Any non-trivial torus action induces a good compatible fibration, which  
is a tangential compatible fibration. 
Moreover the product of two such good compatible fibrations 
is a tangential compatible fibration which is not good in general. 

Though the notion of $G$-tangential compatible fibration can be 
defined for arbitrary compact Lie group $G$, we do not know 
examples of tangential compatible fibration for non-Abelian group.  
\end{example}

Let $M$ be a Riemannian manifold and $W$ a $\Z/2$-graded $Cl(TM\oplus\R^{p,q})$-	module bundle. 
As in the same way in the proof Theorem~7.2 and 
Proposition~7.3 in \cite{Fujita-Furuta-Yoshida2}, 
we have the following.  

\begin{theorem}\label{def of local ind}
Suppose that $V$ is an open subset of $M$ whose complement is compact.  
If $V$ is equipped with a $G$-tangential acyclic compatible system $\{V_{\alpha}, \SF_{\alpha}, D_{\alpha}\}_{\alpha\in A}$, then  
we can define the 
local index $\ind(M, \{V_{\alpha}, \SF_{\alpha}, D_{\alpha}\}_{\alpha\in A},W)=
\ind(M,V,W)=\ind(M,V)$, which satisfies the excision formula, sum formula and product formula. 
\end{theorem}

The resulting index is a representation of $Cl_{p,q}$, 
i.e., an element of the $K$-group with $Cl_{p,q}$-action $K(pt, Cl_{p,q})$, 
where for any locally compact Hausdorff space $X$,  $K$-group 
$K(X, Cl_{p,q})$ is defined as the Grothendieck group of 
the semigroup consisting of equivalence classes of complex vector bundles with 
fiberwise $Cl_{p,q}$-action.  
There is an equivalent description of $K(X,Cl_{p,q)}$. 
Let $\{e_1,\cdots, e_p, e_1',\cdots, e_q'\}$ be the standard basis of $\R^{p,q}$. 
Let $G_{p,q}$ be the finite group generated by 
$\{e_1,\cdots, e_p, e_1',\cdots, e_q', \delta\}$
with relations 
$$
\delta^2=1, \ 
\delta e_i=e_i \delta, \ \delta e_{i'}'=e_{i'}'\delta, \ 
e_i^2=e_{i'}'^2=\delta, \ 
e_ie_je_i^{-1}e_j^{-1}=e_ie_{i'}'e_i^{-1}e_{i'}'^{-1}=
e_{i'}'e_{j'}'e_{i'}'^{-1}e_{j'}'^{-1}=\delta
$$ for $i,j =1,\cdots, p$ ($i\neq j$), $i', j'=1,\cdots, q$ ($i'\neq j'$). 
Note that a representation of $Cl_{p,q}$ is equivalent to a 
representation of $G_{p,q}$ such that $\delta$ acts as $-{\rm id}$. 
In terms of this group, $K(X, Cl_{p,q})$ can be identified with the 
subgroup of $G_{p,q}$-equivariant K-group  $K_{G_{p,q}}(X)$ consisting of 
vector bundles $E$ on which $\delta$ acts as $-{\rm id}$. 

Let us briefly recall the definition of the local index $\ind(M,V,W)$.  
Let $D:\Gamma(W)\to \Gamma(W)$ be a Dirac-type operator. 
% of odd degree. , i.e., $D$ is a formally self-adjoint degree-one operator of order-one whose principal symbol is given by the Clifford action of $Cl(TM)(\subset Cl(TM\oplus \R^{p,q}))$.  
We may assume that $D$ anti-commutes with $Cl_{p,q}(\subset Cl(TM\oplus \R^{p,q}))$-action by taking the average by the finite group $G_{p,q}$. 
We consider 
the perturbation $D+t\sum_{\alpha\in A}\rho_{\alpha}D_{\alpha}\rho_{\alpha}$ for $t\gg 1$, where $\{\rho_{\alpha}\}_{\alpha\in A}$ is a family of smooth cut-off functions 
which is constant along leaves of $\SF_{\alpha}$ and satisfies 
some estimates as in \cite[Subection~4.1]{Fujita-Furuta-Yoshida2}. 
Such a perturbation gives a Fredholm operator on the space of $L^2$-sections of $W$. 
In \cite[Section~7]{Fujita-Furuta-Yoshida2} we explained the details of the definition of $\ind(M,V)$ for good compatible fibrations arising from torus action. 
The construction can be generalised for tangential compatible fibrations 
which is not necessarily good. 
The point is as follows. 
\begin{itemize}
\item We can take a family of smooth cut-off functions $\{\rho_{\alpha}\}_{\alpha\in A}$ which satisfies the following.  
%as in Lemma~\ref{G-inv p.o.u}. 
%\begin{equation}
\begin{itemize}
\item $\rho(x)+\sum_{\alpha}\rho_{\alpha}(x)>0$ for all $x\in M$. 
%If $\rho_{\alpha}^0(x)>0$, then we have $\rho_{\alpha}(x)>0$. 
\item For each $\alpha\in A$ we have $\supp(\rho_{\alpha})\subset V_{\alpha}$. 
\item Each $\rho_{\alpha}$ is $G$-invariant.  
\end{itemize}
\item  We can deform the end of $V$ using a proper function 
 together with $\{V_{\alpha}, \SF_{\alpha}, D_{\alpha}\}_{\alpha\in A}$ so that the end has a cylindrical end and 
the resulting deformed compatible system has translational invariance.  
%\end{equation}
\end{itemize}

\begin{remark}
When $V$ has several connected components $V_{(1)}, V_{(2)}, \cdots, V_{(k)}$ 
we can generalize the definition of $G$-tangential compatible system in such a way that there exists a family of compact Lie groups $\{G_{(1)}, G_{(2)}, \cdots , G_{(k)}\}$ 
and each $G_{(i)}$ acts on the component $V_{(i)}$. 
Moreover using a suitable technical conditions for $\{V_{\alpha}\}_{\alpha\in A}$ it may be possible to generalize in such a way that 
there exists a family of compact Lie groups $\{G_{\alpha}\}_{\alpha\in A}$ and 
each $G_{\alpha}$ acts on $V_{\alpha}$. 
\end{remark}

\begin{remark}
We do not assume that $W$, $D$ and $D_{\alpha}$ are $G$-equivariant. 
We only use the group action to have cut-off functions and 
to deform the end so that the functions are constant along each leaves 
and we can extend the restriction of the data to the cylindrical end. 
\end{remark}

\begin{remark}
All constructions in this section can work for real Clifford module bundles. 
In fact we use real Clifford module bundle case latter. 
\end{remark}

%%%%%%%%%%%%%%%%%%%%%%%%%%%%%%%%%%%%%%%%%%%%%%%%%%%%%%%%%%%%%%%
\section{Cobordism and cobordism invariance of the local index}
\label{Cobordism and cobordism invariance of the local index}

In this section we give the definitions of cobordisms of 
compatible fibrartions and compatible systems in a natural way. 
We also introduce the notion of acyclicity of the cobordism, 
and by using it we give the statement of our main theorem, 
{\it cobordism invariance of the local index}. 

\subsection{Cobordism of compatible fibrations}
\label{Cobordism of compatible fibrations}

Let $V_1$ and $V_2$ be smooth manifolds. 
\begin{definition}\label{coboffib}
Two compatible fibrations 
$\{V_{1,\alpha}, \SF_{1,\alpha}\}_{\alpha\in A}$ on $V_1$ 
and $\{V_{2,\beta}, \SF_{2,\beta}\}_{\beta\in B}$ on $V_2$ 
are {\it cobordant} if there exists a datum 
$\{V,V_{\gamma}, \SF_{\gamma}\}_{\gamma\in C}$ satisfying the 
following conditions. 
\begin{itemize}
\item[(i)] $V$ is a smooth manifold with 
$\partial V=V_1\cup V_2$. 
\item[(ii)] $\{V_{\gamma}, \SF_{\gamma}\}_{\gamma\in C}$ is a compatible fibration 
on $V$. 
\item[(iii)] There exists a tubular neighborhod $V_0$ of $\partial V$ 
which is diffeomorphic to $V_1\times (-\varepsilon,0] \cup 
V_2\times [0,\varepsilon)$ for some small $\varepsilon>0$ and 
 $V_0\cap V_{\gamma}$ consists of leaves of 
$\SF_{\gamma}$ for each $\gamma\in C$. 
\item[(iv)]
Under the diffeomorphism in (iii) 
the restriction of the foliation $\SF_{\gamma}|_{V_0\cap V_{\gamma}}$ for each $\gamma\in C$ 
is isomorphic to the foliation on $V_{1,\alpha}\times (-\varepsilon, 0]\cap V_{\gamma}$ or $V_{2, \beta}\times [0,\varepsilon)\cap V_{\gamma}$ induced from $\SF_{1,\alpha}$ or $\SF_{2,\beta}$ 
for some $\alpha\in A$ or $\beta\in B$.
In particular each leaf of $\SF_{\gamma}|_{V_0\cap V_{\gamma}}$ is 
diffeomorphic to the product of an interval and some leaf of 
$\SF_{1,\alpha}$ or $\SF_{2,\beta}$. 
\end{itemize} 
We call $\{V,V_{\gamma}, \SF_{\gamma}\}_{\gamma\in C}$ a 
{\it cobordism} between 
$\{V_{1,\alpha}, \SF_{1,\alpha}\}_{\alpha\in A}$ 
and $\{V_{2,\beta}, \SF_{2,\beta}\}_{\beta\in B}$. 
\end{definition}

%%%%%%%%%%%%%%%%%%%%%%%%%%%%%%%%%%%%%%%%%%%%%%%%%%%%%%%%%%%%%%%%%%
\subsection{Cobordism of compatible systems}
\label{Cobordism of compatible systems}

%Let $V$ be a smooth Riemannian manifold and 
%$\{V_{\alpha}, \SF_{\alpha}\}_{\alpha\in A}$ a compatible fibration on $V$. 
%Let $W\to V$ be a $\Z/2$-graded $Cl(TV)$-module bundle over $V$. 

Let $V_1$ and $V_2$ be smooth Riemannian manifolds. 
For $i=1,2$ let $W_i\to V_i$ be $\Z/2$-graded $Cl(TV_i\oplus \R^{p,q})$-module bundles. 
\begin{definition}\label{cobofsys}
Two compatible systems 
$\{V_{1,\alpha}, \SF_{1,\alpha}, W_1, D_{1,\alpha}\}_{\alpha\in A}$ on $V_1$ and 
$\{V_{2,\beta}, \SF_{2,\beta}, W_2, D_{2,\beta}\}_{\beta\in B}$ on $V_2$ are 
{\it cobordant} if there is a collection of data 
$\{V,V_{\gamma}, \SF_{\gamma}, W, D_{\gamma}\}_{\gamma\in C}$ 
which satisfies the following conditions. 
\begin{itemize}
\item[(i)] $\{V,V_{\gamma}, \SF_{\gamma}\}_{\gamma\in C}$ is a 
cobordism between $\{V_{1,\alpha},\SF_{1,\alpha}\}_{\alpha\in A}$ and 
$\{V_{2,\beta}, \SF_{2,\beta}\}_{\beta\in B}$. 
\item[(ii)] The diffeomorphism in Definition~\ref{coboffib}(iii) 
is isometric. 
\item[(iii)] $W$ is a $Cl(TV\oplus \R^{p,q})$-module bundle over $V$ 
(without $\Z/2$-grading). 

%\item[(iv)] $D_{\gamma}$ is a Dirac-type operator along leaves of $\SF_{\gamma}$ for each $\gamma$, i.e., $D_{\gamma}$ is an order-one formally self-adjoint operator whose principal symbol is given by the Clifford action of the directions of leaves. 
%, i.e., 
%$D_{\gamma}$ is a formally self-adjoint first order differential operator 
%which acts on $\Gamma(W|_{V_{\gamma}})$ and 
%contains only derivatives along leaves of $\SF_{\gamma}$ whose 
%symbol is given by the Clifford multiplication of the direction of 
%leaves. 
%\item[(v)] 
%For a leaf $L\in \SF_\gamma$ let $\tilde u\in \Gamma (\nu_\gamma|_L)$ be a section of $\nu_\gamma|_L$ parallel along $L$. 
%For $b\in U_{\alpha}$ and $u\in T_bU_{\alpha}$, let $\tilde{u}\in \Gamma(TV_{\alpha}|_{\pi^{-1}_{\alpha}(b)})$ be the horizontal lift of $u$ with respect to the Riemannian metric $g|_{V_\alpha}$. 
%$\tilde{u}$ acts on $W|_L$ by the Clifford multiplication $c(\tilde{u})$. Then $%D_{\gamma}$ and $c(\tilde{u})$ anti-commute each other. 
%, i.e. 
%\[
%0=\{ D_{\alpha},c(\tilde{u}) \}:=
%D_{\alpha}\circ c(\tilde{u})+c(\tilde{u})\circ D_{\alpha}
%\]
%as an operator on $W|_L$. 
\item[(iv)] $\{D_{\gamma}\}_{\gamma\in C}$ is a compatible system on $(\{V_{\gamma}, \SF_{\gamma}\}_{\gamma}, W)$. 
\item[(v)]
$W|_{V_0}\to V_0$ is isometric to 
$W_1\times(-\varepsilon, 0]\cup W_2\times [0,\varepsilon)
\to V_1\times(-\varepsilon, 0]\cup V_2\times [0,\varepsilon)$, 
where $V_0$ is the tubular neighborhood of 
$\partial V$ as in Definition~\ref{coboffib}(ii). 
\item[(vi)] Under the identification in (v) we have 
$c(dr)|_{W_i^{\pm}}=\pm\sqrt{-1}$ for $i=1,2$, where $c$ is 
the Clifford multiplication of $W$ and $r$ is the standard 
coordinate of $(-\varepsilon, 0]\cup[0,\varepsilon)$. 
\item[(vii)] Under the identification  in (v) 
$D_{\gamma}|_{V_0\cap V_{\gamma}}$ for each $\gamma$ 
is identified with $D_{1,\alpha}$ or $D_{2,\beta}$ 
for some $\alpha$ or $\beta$. 
\end{itemize}
We call the data $\{V,V_{\gamma}, \SF_{\gamma}, W, D_{\gamma}\}_{\gamma\in C}$ 
a {\it cobordism between} 
$\{V_{1,\alpha}, \SF_{1,\alpha}, W_1, D_{1,\alpha}\}_{\alpha\in A}$ and 
$\{V_{2,\beta}, \SF_{2,\beta}, W_2, D_{2,\beta}\}_{\beta\in B}$. 
\end{definition}

\begin{definition}\label{acysyscob}
A cobordism $\{V,V_{\gamma}, \SF_{\gamma}, W, D_{\gamma}\}_{\gamma\in C}$  is called {\it acyclic} if 
the compatible system $\{D_{\gamma}\}_{\gamma\in C}$ is acyclic. 
%the following conditions are satisfied. 
%\begin{itemize}
%\item[(i)] $\ker(q_L^*D_{\gamma}|_{\pi^{-1}(\tilde b)})=\{0\}$ 
%for each $\gamma\in C$, $L\in\SF_{\gamma}$ and $b\in \tilde U_{\gamma}$. 
%\item[(ii)] If $V_{\gamma}\cap V_{\gamma'}\neq\emptyset$ for $\gamma, \gamma'\in C$, 
%then $D_{\gamma}D_{\gamma'}+D_{\gamma'}D_{\gamma}\geq 0$. 
%\end{itemize}
\end{definition}

%%%%%%%%%%%%%%%%%%%%%%%%%%%%%%%%%%%%%%%%
\subsection{Cobordism invariance of local index}
\label{Cobordism invariance of local index} 
%Consider the following setting. 
%\\
Let $M_i$ be Riemannian manifolds and 
$W_i$ $\Z/2$-graded $Cl(TM_i\oplus\R^{p,q})$-module bundles over $M_i$ for $i=1,2$. 
Let $M$ be a Riemannian manifold and 
$W$ a $Cl(TM\oplus \R^{p,q})$-module bundle over $M$ 
without $\Z/2$-grading.  
%with $\partial\tilde M=M$. Let  
Suppose that the following conditions are satisfied. 

\medskip

\begin{itemize}
\item[(i)] $\partial M=M_1\cup M_2$. 
\item[(ii)] For each $i=1,2$ there exists an open subset $V_i$ of $M_i$ which 
has compact complement and 
is equipped with an acyclic $G$-tangential compatible systems 
$\{V_{i,\alpha}, \SF_{i,\alpha}, W|_{V_i}, D_{i,\alpha}\}$ 
for a compact Lie group $G$. 
\item[(iii)] There exists an open subset $V$ of $M$ which 
has compact complement and is 
equipped with a $G$-tangential acyclic compatible system 
$\{V_{\alpha}, \SF_{\alpha}, W|_{V}, D_{\alpha}\}$. 
\item[(iv)] $\{V,V_{\gamma}, \SF_{\gamma}, 
W|_{V}, D_{\gamma}\}$ is a cobordism between 
$\{V_{1,\alpha}, \SF_{1,\alpha}, W|_{V_1}, D_{1,\alpha}\}$ and 
$\{V_{2,\beta}, \SF_{2,\beta}, W|_{V_2}, D_{2,\beta}\}$. 
\item[(v)] The isometric in Definition~\ref{coboffib}(ii) is 
$G$-equivariant. 
\end{itemize}

\medskip

%\begin{theorem}
%\label{cobinv}
%Under the above assumption if $M_2=\emptyset$, then we have 
%$$
%\ind(M_1,V_1)=0. 
%$$
%\end{theorem}

The following is a main theorem in this paper, which will be 
reformulated and shown in the subsequent sections. 
\begin{theorem}\label{maintheorem}
Under the above setting we have 
$$
\ind(M_1,V_1)=\ind(M_2,V_2). 
$$
\end{theorem}

%%%%%%%%%%%%%%%%%%%%%%%%%%%%%%%%%%%%%%%%%%%%%%%%%%%%%%%%%%%%%%%%%%%%%%%%%
\section{Cobordism invariance and extra symmetry of Clifford algebras}
\label{Cobordism invariance via extra symmetry of Clifford algebras} 
In this section we give a reformulation of the cobordism invariance 
by using an extra symmetry of Clifford algebras for indefinite metrics. 
In Subsection~\ref{Complex Clifford module bundle case} 
we give the reformulation for the 
complex case (Theorem~\ref{cobinvref}). 
We also give in Subsection~\ref{Real Clifford module bundle case} 
a formulation of cobordism invariance of local index for 
real Clifford module bundles (Theorem~\ref{cobinvreal}).  
We will show the real case in the next section.  
In this section we consider the following set-up. 

\medskip

\noindent
{\bf Set-up.} \ 
$M$ is a Riemannian manifold and 
$V$ is an open subset of $M$. 
Suppose that $V$ is equipped with a $G$-tangential 
compatible fibration $\{V_{\gamma}, \SF_{\gamma}\}_{\gamma\in\hat A}$ 
for a compact Lie group $G$. 
Moreover we assume that there exists a smooth function $f:M\to \R$ 
such that $f|_{M\setminus V}$ is proper and $f|_V$ is $G$-invariant.

\subsection{Complex Clifford module bundle case}
\label{Complex Clifford module bundle case} 
Let $\hat W\to M$ be 
a {\it complex} $Cl(TM\oplus\R e_L\oplus\R^{p,q}\oplus\R e_R\oplus\R e_R')$-module bundle, 
where the metric is defined by 
$e_L\cdot e_L=1$ and $e_R\cdot e_R=e_R'\cdot e_R'=-1$. 
Suppose that there exists a $G$-tangential acyclic compatible system 
$\{V_{\gamma}, \SF_{\gamma}, \hat W|_V, \hat D_{\gamma}\}_{\gamma\in \hat A}$ on $V$. 
Let $(R_{\C},c_{1,2})$ be an irreducible complex representation of 
$Cl(\R e_L\oplus \R e_R\oplus \R e_R')=Cl_{1,2}$ 
defined by $R_{\C}:=\wedge^{\bullet}\C=\C\oplus\C$ and 
$$
c_{1,2}(e_L)=\begin{pmatrix} 0 &  -1 \\ 1 & 0 
\end{pmatrix}, \quad 
c_{1,2}(e_R)=\begin{pmatrix} 0 &  1 \\ 1 & 0
\end{pmatrix}, \quad 
c_{1,2}(e_R')=\begin{pmatrix} 1 & 0 \\ 0 & -1
\end{pmatrix}. 
$$
We define $Cl(TM\oplus \R^{p,q})$-module bundle 
$$
W:={\rm Hom}_{Cl_{1,2}}(R_{\C},\hat W)\to M. 
$$ 
Note that we have 
$W\otimes R_{\C}\cong \hat W$ as $Cl(TM\oplus \R^{p,q})\otimes Cl_{1,2}\cong 
Cl(TM\oplus \R e_L\oplus \R^{p,q}\oplus\R e_R\oplus\R e_R')$-module bundles. 
For each regular value $r$ of $f$, we put $M_r:=f^{-1}(r)$ and 
$V_r:=M_r\cap V$. 
Then $W_r:=W|_{M_r}$ has a structure of $Cl(TM_r\oplus \R v_r\oplus \R^{p,q})$-module bundle, where 
$v_r$ is the unit tangent vector which is normal to $TM_r$ with the increasing direction of $f$. 
Since we have $c(v_r)^2=-1$, we can define $\Z/2$-grading $W_r=W_r^+\oplus W_r^-$ by $c(v_r)|_{W_r^{\pm}}=\pm \sqrt{-1}$. 
For each $\alpha\in \hat A$ the Dirac-type operator along leaves $\hat D_{\alpha}$ induces an another Dirac-type operator along leaves $D_{\alpha}:\Gamma(W|_{V_{\alpha}})\to\Gamma(W|_{V_\alpha})$ defined by 
$(D_{\alpha}s)(r):=\hat D_{\alpha}(s(r))$ for $s\in \Gamma(W)$ and $r\in R$. 
Let $A_r$ be the subset of $\hat A$ defined by 
$A_r=\{\alpha\in\hat A \  | \ V_{\alpha}\cap M_r\neq \emptyset\}$. 
For $\alpha\in A_r$ we put $V_{r,\alpha}=V_{\alpha}\cap V_r$ and 
define $D_{r,\alpha}:\Gamma(W|_{V_{r,\alpha}})\to 
\Gamma(W|_{V_{r,\alpha}})$ by the restriction 
of $D_{\alpha}$. 
Since $f|_V$ is $G$-invariant we have the $G$-tangential compatible 
fibration $\{\SF_{r,\alpha}\}_{\alpha\in A_r}$ and compatible system 
$\{D_{r,\alpha}\}_{\alpha\in A_r}$ on $V_r=M_r\cap V$. 
The following is clear by definition of $M_r$ and $D_{\alpha}$. 

\begin{lemma}\label{induced comp sys}
The $G$-tangential compatible system 
$\{V_{r,\alpha}, \SF_{r,\alpha}, W_r, D_{r,\alpha}\}_{\alpha\in A_r}$ 
over $V_r$ is acyclic. 
\end{lemma} 
Since $M_r\setminus V_r$ is compact  
the local index $\ind(M_r, V_r,W_r)$ can be defined 
as an element of the $K$-group with $Cl_{p,q}$-action 
$K(pt,Cl_{p,q})$. 

\begin{theorem}\label{cobinvref}
The local index $\ind(M_{r}, V_{r},W_{r})$ does not depend on the 
choice of the regular value $r$. 
\end{theorem}

Given a $Cl(TM\oplus \R^{p,q})$-module bundle $W$ as in Theorem~\ref{maintheorem}, we put $\hat W:=W\otimes R_{\C}$, then we have  
${\rm Hom}_{Cl_{p,q}}(R_{\C}, \hat W)\cong W$, and hence, 
Theorem~\ref{maintheorem} is 
equivalent to the above Theorem~\ref{cobinvref}.

%\begin{color}{red}
%$\hat D_{\alpha}$'Æ—]•ª'ÈCliffordì—p'̐®‡«? 
%\end{color}

%%%%%%%%%%%%%%%%%%%%%%%%%%%%%%%%%%%%%%%%%%%%%%%%%%%%%%%%%%%%%%%
\subsection{Real Clifford module bundle case}
\label{Real Clifford module bundle case}
Let $\hat W$ be a {\it real} $Cl(TM\oplus\R^{p,q}\oplus\R e_0\oplus \R e_1\oplus\R\epsilon)$-module bundle over $M$, where the metric is defined by 
$e_0\cdot e_0=e_1\cdot e_1=\epsilon\cdot\epsilon=-1$. 
We regarded $\hat W$ as a $\Z/2$-graded $Cl(TM\oplus\R^{p,q}\oplus\R e_0\oplus\R e_1)$-module by using the Clifford multiplication $c(\epsilon)$. 
%Then $W={\rm Hom}_{Cl_{p,q}}(R,\hat W)\to M$ has a structure of $CL(TM)$-module bundle with the Clifford multiplication $c$, and 
%$W_0=W|_{M_0}=W_0^+\oplus W_0^-$ has a structure of $\Z/2$-graded $Cl(TM)$-module bundle whose grading is given by the eigenspace decomposition of $c(v_0)$. 
%
Suppose that there exists a $G$-tangential acyclic compatible systems 
$\{V_{\gamma}, \SF_{\gamma}, \hat W|_{V}, \hat D_{\gamma}\}_{\gamma\in \hat A}$ on $V$. 
Let $r$ be a regular value of $f$. 
When we put $M_r:=f^{-1}(r)$, we have the 
$\Z/2$-graded $Cl(TM_r\oplus \R v_r\oplus \R^{p,q}\oplus \R e_0\oplus \R e_1)$-module bundle $\hat W_r:=\hat W|_{M_r}$ as in Subsection~\ref{Complex Clifford module bundle case}. 
%, where $v_r$ is the unit normal tangent vector to $TM_r$ with the increasing direction of $f$. 
%]
Let $(R,c_{1,2})$ be a realification of the complex representation $R_{\C}$ 
of $Cl(\R v_r\oplus \R e_0\oplus \R e_1)=Cl_{1,2}$ as in 
Subsection~\ref{Complex Clifford module bundle case}, 
which is an irreducible real representation of $Cl_{1,2}$. 
% and equipped with a complex structure $J$ which commutes with $Cl_{1,2}$-action. 
Namely we put $R:=\R\oplus \R$ and define 
$c_{1,2}:Cl_{1,2}\to {\rm End}(R)$ by 
\begin{equation}\label{realirrrep}
c_{1,2}(v_r)=\begin{pmatrix} 
 0 & -1 \\
1 & 0 
\end{pmatrix}, \ 
c_{1,2}(e_0)=\begin{pmatrix} 
0 & 1 \\
1 & 0 
\end{pmatrix}, \ 
c_{1,2}(e_1)=\begin{pmatrix} 
1 & 0 \\
0 & -1 
\end{pmatrix}.
\end{equation}
We define a $\Z/2$-graded $Cl(TM_r\oplus \R^{p,q})$-module bundle $W_r$ by 
$W_r:={\rm Hom}_{Cl_{1,2}}(R,\hat W_r)$. 
Note that we have $W_r\otimes R\cong\hat W_r$ 
as $Cl(TM_r\oplus\R^{p,q})\otimes Cl_{1,2}\cong 
Cl(TM_r\oplus\R v_r\oplus\R^{p,q}\oplus\R e_0\oplus \R e_1)$-module bundles. 
%Note  that there exists an induced 
%$G$-tangential compatible systems 
%$\{V_{\gamma}, \SF_{\gamma}, W|_{V}, D_{\gamma}\}_{\gamma\in \hat A}$. 
%
When we put $V_r:=M_r\cap V$ we have the induced $G$-tangential acyclic compatible system 
$\{V_{r,\alpha}, \SF_{r,\alpha}, W_r, D_{r,\alpha}\}_{\alpha\in A_r}$ 
on $V_r$. 
Since $M_r\setminus V_r$ is compact 
we have the local index $\ind(M_r,V_r,W_r)$ as an element of 
the real $K$-group with $Cl_{p,q}$-action $KO(pt, Cl_{p,q})$, 
where for any locally compact Hausdorff space $X$, 
the group $KO(X,Cl_{p,q})$ is defined as the Grothendieck group 
of the semi-group consisting of real vector bundles over $X$ 
equipped with a fiberwise action of $Cl_{p,q}$. 
%Then we have the cobordism invariance of the local index as follows. 

Under the above set-up we have the following cobordism invariance for real Clifford module bundles, which will be shown in the next section.  
\begin{theorem}\label{cobinvreal}
The local index $\ind(M_{r}, V_{r},W_{r})$ does not depend on the 
choice of the regular value $r$. 
\end{theorem}
%
%Since the proof is same as that of Theorem~\ref{cobinvref}, we omit it. 
%The points are the following three. 
%\begin{enumerate}
%\item $R_{\R}$ is a realification of a irreducible complex representation 
%$(R, J_R)$ of $Cl_{1,2}$. 
%{\color{red} $R_{\R}=\C\oplus \C$?}
%\item We can consider the perturbation $\bD_{t,a}$ as in Subsection~\ref{The perturbation of Dirac-type operator} which commutes with $J_R$. 
%\end{enumerate}

%\medskip

%%%%%%%%%%%%%%%%%%%%%%%%%%%%%%%%%%%%%%%%%%%%%%%%%%%%%%%%%%%%%%%%%%%%%%%%
\section{Proof of Theorem~\ref{cobinvreal}}
\label{Proof of Theorem{realcobinv}}
We prove Theorem~\ref{cobinvreal} by using the technique 
developed in \cite{Braverman,Bravermancobinv}, which is a version of 
Witten-type deformation. 
We may assume that $r_0=0$ for simplicity. 
By the additivity of the index, it suffices to consider the case 
$M_{r_0}=M_0=\partial M$. 
For later convenience let us recall the set-up and fix notations. 

\medskip

\noindent
{\bf Set-up.}  Let $M_0$ and $M$ be Riemannian manifolds 
such that $\partial M=M_0$. 
Let $\hat W$ be a real $Cl(TM\oplus \R^{p,q}\oplus\R e_0\oplus\R e_1\oplus\R\epsilon)$-module bundle over $M$. 
We regard $\hat W$ as a $\Z/2$-graded $Cl(TM\oplus\R^{p,q}\oplus\R e_0\oplus\R e_1)$-module by using the Clifford multiplication of $c(\epsilon)$. 
%Then $W={\rm Hom}_{Cl_{p,q}}(R,\hat W)\to M$ has a structure of $CL(TM)$-module bundle with the Clifford multiplication $c$, and 
%$W_0=W|_{M_0}=W_0^+\oplus W_0^-$ has a structure of $\Z/2$-graded $Cl(TM)$-module bundle whose grading is given by the eigenspace decomposition of $c(v_0)$. 
%Our assumption is the following. 
%\begin{assump}
Suppose that there exists an open subset $V$ of $M$ which has compact complement  
and is equipped with an acyclic $G$-tangential compatible systems 
$\{V_{\gamma}, \SF_{\gamma}, \hat W|_{V}, \hat D_{\gamma}\}_{\gamma\in \hat A}$. 
%\end{assump}
When we put $V_0:=M_0\cap V$ 
we have the induced $G$-tangential acyclic compatible system 
$\{V_{0,\alpha}, \SF_{0,\alpha}, W_0, D_{0,\alpha}\}_{\alpha\in A}$ 
over $V_0$, 
where $W_0=W|_{M_0}=W_0^+\oplus W_0^-$ is  the $\Z/2$-graded $Cl(TM_0)$-module. 
% whose grading is given by 
%the eigenspace decomposition of 
%the Clifford action $c(\epsilon)$. 

\medskip

Though we do not assume that $M$ and $M_0$ are complete, 
we can deform them so that they are complete. 
For instance we can deform them into complete manifolds with cylindrical ends.  
Moreover we can extend the given data, e.g., compatible system, 
to the data on the completion so that they have translational invariance on the end. 
Hereafter we may assume that $M$ and $M_0$ are complete 
and the given data have uniformity on the end by using such completion. 
By the excision property, the local indices do not change 
under the completion. 

\subsection{The perturbation of Dirac-type operator}
\label{The perturbation of Dirac-type operator} 
%\subsection{The Clifford module bundle structure}
Let $\tilde M$ be the manifold obtained from $M$ by attaching the 
cylinder $M_0\times [0,\infty)$, i.e., 
$\tilde M=M\cup M_0\times[0,\infty)$. 
Then we can extend 
$\{V, V_{\gamma}, \SF_{\gamma}, 
\hat W|_{ V}, \hat D_{\gamma}\}_{\gamma\in \hat A}$ naturally to 
$\tilde M$, and we denote them by 
$\{\tilde V,\tilde V_{\gamma}, \tilde\SF_{\gamma}, 
\tilde W|_{\tilde V}, \tilde D_{\gamma}\}_{\gamma\in \hat A}$. 
%We put $\hat W'=\hat W\otimes \wedge^{\bullet} \C$ and 
%define $\hat c':T^*\hat M\to {\rm End}(\hat W')$ by 
%$$
%\hat c'(v)=\sqrt{-1}c(v)\otimes c_L(1) \quad (v \in T^*\hat M=T\hat M),  
%$$where $c_L:\R\to {\rm End}(\wedge^{\bullet}\C)$ is defined by 
%using the 
%contraction $\iota_{r}$ on $\wedge^{\bullet}\C$  for $r\in \R$ 
%$$c_L(r)(\omega)=r\wedge\omega-\iota_r(\omega) \quad (\omega\in \wedge^{\bullet}\C).$$ 
%The map $\hat c'$ gives a $\Z/2$-graded $Cl(T\hat M)$-module bundle structure on $\hat W'$. 
%\begin{lemma}
%Put $\hat{(W')}^{+}=\hat W\otimes\wedge^0 \C$ and 
%$\hat{(W')}^{-}=\hat W\otimes\wedge^1 \C$. 
%The map $\hat c'$ gives a structure of $\Z/2$-graded 
%$Cl(T\hat M')$-module bundle. 
%\end{lemma}
%
We put $W_0={\rm Hom}_{Cl_{1,2}}(R,\tilde W|_{M_0})$ for an irreducible real representation $R$ of $Cl_{1,2}$.   
We take and fix a $Cl_{p,q}$ anti-invariant Dirac-type operator $D:\Gamma(W_0)\to \Gamma(W_0)$. 
% , i.e., $D$ is a degree one self adjoint operator of order one whose principal symbol is given by the Clifford action of $Cl(TM_0)(\subset Cl(TM_0\oplus \R^{p,q}))$ and anti-commutes with $Cl_{p,q}$-action. 
%
Let $\tilde D:\Gamma(\tilde W)\to\Gamma(\tilde W)$ be a Dirac-type operator such that 
\begin{equation}\label{initialDirac}
\tilde D|_{M_0\times(0,\infty)}=\pi^*D\otimes 1+1\otimes c(v_0)\partial_r
\end{equation}
under the isomorphism $\tilde W|_{M_0\times (0,\infty)}\cong 
\pi^*\tilde W|_{M_0}\cong\pi^*(W_0\otimes R) $, 
where 
$\pi$ is the natural projection to the first factor of $M\times (0,\infty)$, 
$\partial_r$ is the covariant derivative of $(0,\infty)$-direction with respect to 
a Clifford connection on $\pi^*\tilde W|_{M_0}$. 
%We also define $\tilde D':\Gamma(\tilde W)\to \Gamma(\tilde W)$ by 
%$\tilde D':=\tilde D\otimes c_{p+1,q+2}(e_L)$ using the 
%isomorphism $W\otimes R=\hat W$. 
We take a family of cut off functions $\{\rho_{\gamma}^2\}$ for $\tilde M=(\tilde M\setminus \tilde V)\cup \left(\cup_{\gamma}\tilde V_{\gamma}\right)$ 
as in the end of Subsection~\ref{ind(M,V,W)}. 
%Here we consider the trivial foliation on a neighborhood of $\hat M\setminus \hat V$. 
Let $D_{\gamma}$ be the Dirac-type operator along leaves which is induced by 
$\tilde D_{\gamma}$ and we put $D_{\gamma}':=\rho_{\gamma}D_{\gamma}\rho_{\gamma}$. 
For $t>0$ we put 
%$\tilde D_{\gamma}'=\rho_{\gamma}\tilde D_{\gamma}\rho_{\gamma}$ and 
\begin{equation}\label{perturbation1}
\tilde D'_t=
\left(\tilde D+t\sum_{\gamma} D_{\gamma}'\right)c(e_0).  
\end{equation}
%where we use the isomorphism $\tilde W\cong W\otimes R$. 
We take and fix a smooth function 
$s:\R\to [0,\infty)$ and $p:\tilde M\to\R$ such that 
%$s(r)=\pm 4$ for $|r|\geq 3$, 
$s(r)=r$ for $|r|\geq 1$, $s(r)=0$ for $|r|\leq 1/2$, 
$p(x)=0$ for $x\in M\subset \tilde M$ and $p(y,r)=s(r)$ for 
$(y,r)\in M_0\times (0,\infty)$. 
For $a\in \R$ we put 
\begin{equation}\label{perturbation2}
{\mathbf D}_{t,a}=\tilde D'_t-c((p(x)-a)e_1). 
\end{equation}
%where $c_R:\R\to {\rm End}(\wedge^{\bullet}\C)$ is defined by 
%$c_R(r)(\omega)=r\wedge\omega+\iota_r(\omega)$ for $r\in \R$. 
%Note that $c_R(r)^2=r^2$ and $c_R(r)c_L(r)+c_R(r)c_L(r)=0$ for any $r\in\R$. 

\begin{prop}\label{laplacian}
For any $t>0$ and $a\in\R$ we have 
$$
\bD_{t,a}^2=(\tilde D'_t)^2+B+|p(x)-a|^2, 
$$where $B$ is a uniformly bounded bundle map of $\tilde W$ of degree 0.  
%which vanishes on $M_0\times [4,\infty)\subset \tilde M$.  
%which satisfies $B|_{\tilde M}=0$ and 
%$B|_{M\times (0,\infty)}=\sqrt{-1}c_r(\Pi_1-\Pi_0)$ for the 
%projections $\Pi_i : \hat W'|_{M\times (0,\infty)}\to 
%\hat W|_{M\times (0,\infty)}\otimes \wedge^i\C$ ($i=0,1$). 
\end{prop}
\begin{proof}
The proof is almost same as that for \cite[Lemma~10.4]{Braverman}.  
%by putting $B:=\sqrt{-1}s'c_r\otimes 1(\Pi_1-\Pi_0)$. 
The point is $\left(\pi^*D+t\sum D_{\gamma}'\right)c(e_0)$ anti-commutes with $c((s(r)-a)e_1)$ on $\Gamma(\tilde W|_{M_0\times (0,\infty)})$. 
\end{proof}

\begin{prop}\label{laplacianestimate}
For any $\lambda>0$ 
there exists  $T>0$ such that for  each $a\in\R$ there exists an open subset 
$V(a)$ of $\tilde M$ whose complement is compact and for each compactly supported section $\phi\in \Gamma(\tilde W)$ with ${\rm supp}(\phi)\subset V(a)$ we have 
$$
\|\bD_{t,a}\phi\|_{L^2(\tilde W)}\geq \lambda\|\phi\|_{L^2(\tilde W)}. 
$$
\end{prop}
\begin{proof}
We take a pre-compact open neighborhood $U$ of $M\setminus V$ in $M$.  
Let $K$ be the complement $M\setminus U$ and put $K_0:=K\cap M_0$. 
For each $a\in\R$ we take $b>0$ large enough so that we have 
$$
-\|B\|_{\infty}+|p(x)-a|^2\geq \lambda
$$
on $K(a):=(M_0 \setminus U)\times [b,\infty)$.  
We define an open subset $V(a)$ of $\tilde M$ by the interior part of the union $K_0\times[0,\infty) \cup K(a)$. 
By definition the complement $\tilde M\setminus V(a)=(M\setminus V)\cup (M_0\setminus U)\times [0,b]$ is compact. 
As in the proof of \cite[Proposition~{4.4}]{Fujita-Furuta-Yoshida2}, 
there exists a constant $C>0$ such that 
$$
\|\hat D_t'\phi\|^2_{L^2(\tilde W|_{V})}\geq (t^2-Ct)\|\phi\|^2_{L^2(\tilde W|_V)}, 
$$
for any compactly supported section $\phi\in\Gamma(\tilde W)$ with 
${\rm supp}(\phi)\subset V$, and hence we have 
$$
\|\hat D_t'\phi\|^2_{L^2(\tilde W|_{K})}\geq (t^2-Ct)\|\phi\|^2_{L^2(\tilde W|_K)}, 
$$
for any compactly supported section $\phi\in\Gamma(\tilde W)$ with 
${\rm supp}(\phi)\subset K$. 
%$K$'©'çcut-off ŠÖ"'Å$V$ã'Ìcpt'äØ'f'Æ'Ý'È'¹'΂悢. 
Let $\phi\in \Gamma(\tilde W)$ be a compactly supported section with ${\rm supp}(\phi)\subset V(a)$. 
We take $T>0$ large enough so that we have 
$t^2-Ct-\|B\|_{\infty}\geq \lambda$ for any $t>T$. 
Together with Proposition~\ref{laplacian} we have  
\begin{eqnarray*}
\|\bD_{t,a}\phi\|^2_{L^2(W|_{V(a)})}&=&
\int_{V_a}(((D_t')^2+B+|p(x)-a|^2)\phi,\phi))\\
&=&\left(\int_{K_0\times [0,\infty)}+\int_{K(a)}\right)(((D_t')^2+B+|p(x)-a|^2)\phi,\phi))
\\
&\geq&
\int_{K_0\times [0,\infty)}(((D_t')^2+B)\phi,\phi))
+\int_{K(a)}(B+|p(x)-a|^2)\phi,\phi))\\
&\geq& (t^2-Ct-\|B\|_{\infty})\|\phi\|^2_{L^2(W|_{K_0\times[0,\infty)})}
 +\lambda\|\phi\|^2_{L^2(W|_{K(a)})}\\
&\geq&\lambda\|\phi\|^2_{L^2(W|_{K_0\times[0,\infty)})}
 +\lambda\|\phi\|^2_{L^2(W|_{K(a)})}
=\lambda\|\phi\|^2_{L^2(W|_{V(a)})}.
\end{eqnarray*}
\end{proof}

\begin{prop}\label{indDta}
For any $a\in\R$ and $T>0$ as in Proposition~\ref{laplacianestimate} 
the space of $L^2$-solutions $\ker_{L^2}(\bD_{t,a})$ is a $\Z/2$-graded 
finite-dimensional vector space 
and its super-dimension $\ind(\bD_{t,a})$ does not depend on $t>T$ and $a$. 
\end{prop}

\begin{proof}
By \cite[Theorem~3.2]{Fujita-Furuta-Yoshida2}, 
$\ker_{L^2}(\bD_{t,a})$ is a $\Z/2$-graded 
finite-dimensional vector space for each $t>T$ and $a\in\R$. 
Moreover its super-dimension $\ind(\bD_{t,a})$ does not depend on $t>T$ and 
$a\in \R$ by the homotopy invariance and the excision property. 
\end{proof}

\begin{prop}\label{indDta=0}
Let $T>0$ be the constant as in Proposition~\ref{indDta}. 
For any $t>T$ and $a\in \R$ we have 
$$
\ind(\bD_{t,a})=\dim\ker_{L^2}(\bD_{t,a})^+-\dim\ker_{L^2}(\bD_{t,a})^-=0. 
$$
\end{prop}
\begin{proof}
The proof is same as that for \cite[Proposition~10.7]{Braverman}. 
\end{proof}

\subsection{The model operator on the cylinder} 
Consider the natural extension of the $\Z/2$-graded Clifford module bundle 
$\tilde W|_{M_0\times (0,\infty)}$ to $M_0\times \R$. 
We denote it by the same notation $\tilde W$. 
%, and then we have 
%$\hat W'=\pi^*W\otimes \wedge^{\bullet}\C$ for the projection $\pi:M\times \R\to M$. 	   
%For a Dirac type operator $D$ on $W_0\to M_0$ we can define a 
%self-adjoint operator of odd degree, 
%$$
%\tilde D:=\sqrt{-1}(\pi^*D+c(v_r)\partial_r)\otimes c_{p+1,q+2}(e_L) 
%$$on $\tilde W=W_0\otimes R$, where the $\Z/2$-grading is given by the 
%$c_{p+1, q+2}(\epsilon)=\pm 1$. 
We put $W_0:={\rm Hom}_{Cl_{1,2}}(R,\tilde W)$ for the irreducible 
representation $R$ of $Cl_{1,2}$ as in (\ref{realirrrep}). 

Let $D:\Gamma(W_0')\to \Gamma(W_0')$ be a Dirac-type operator of odd degree, 
where the $\Z/2$-grading $W_0'=(W_0')^+\oplus (W_0')^-$ is given by 
$c(\epsilon)=\pm 1$.  
%We have the $G$-tangential acyclic compatible system 
%$\{D_{0,\alpha}\}_{\alpha\in A}$ on $V_0=M_0\cap V$ as follows. 
%Let $A$ be the subset of $\hat A$ defined by 
%$A=\{\alpha\in\hat A \  | \ V_{\alpha}\cap M_0\neq \emptyset\}$. 
%For $\alpha\in A$ we put $V_{0,\alpha}=V_{\alpha}\cap V_0$ and 
%define $D_{0,\alpha}:\Gamma(W|_{V_{0,\alpha}})\to 
%\Gamma(W|_{V_{0,\alpha}})$ by the restriction 
%of $D_{\alpha}$. 
%Following \cite{Braverman} and \cite{Shubin} we define the model operator. 

\begin{definition}
For $t>0$ we define the {\it model operator} $\bD^{\rm mod}_t$ 
which acts on the space of smooth sections of $\tilde W\to M_0\times \R$ by 
\begin{equation}\label{modelop1}
\bD^{\rm mod}_t=\left(r^*D+c(v_r)\partial_r+t\sum_{\alpha} D_{0,\alpha}'\right)c(e_0)-c(r(x)e_1) 
\end{equation}
under the isomorphism $W_0'\otimes R\cong \tilde W$, 
where $r:M_0\times \R\to \R$ is the projection and we put 
$D_{0,\alpha}'=\rho_{\alpha}D_{0,\alpha}\rho_{\alpha}$.  
%is the Dirac-type operator 
%along leaves which is obtained by the restriction of $\{D_{\gamma}'\}_{\gamma\in \hat A}$. 
\end{definition}

As in the same way in Proposition~\ref{indDta}, there exists $T^{\rm mod}>0$ such that $\ker_{L^2}(\bD_{t}^{\rm mod})$ is a finite-dimensional vector space 
and its index $\ind(\bD_{t}^{\rm mod})$ can be defined for any $t>T^{\rm mod}$ as 
its super-dimension. 

\begin{prop}\label{kermod=kerDt}
As a $\Z/2$-graded vector space $\ker_{L^2}(\bD_{t}^{\rm mod})^2$ is isomorphic to 
$\ker_{L^2}(D+t\sum_{\alpha}D_{0,\alpha}')^2\otimes \R^2$ for any $t>T^{\rm mod}$. 
In particular we have 
$$
\ind(\bD_t^{\rm mod})=2\ind(D+t\sum_{\alpha}D_{0,\alpha}')=2\ind(M_0, V_0). 
$$
\end{prop}
\begin{proof}
Since $\left(r^*D+t\sum_{\alpha} D_{0,\alpha}'\right)c(e_0)$ and 
$c(v_r)c(e_0)\partial_r-c(r(x)e_1)$ anti-commutes and $c(v_r)c(e_0)=-c(e_1)$
we have 
\begin{eqnarray*}
(\bD_t^{\rm mod})^2&=&\left(r^*D+t\sum_{\alpha}D_{0,\alpha}'\right)^2+\left(c(v_r)\partial_r c(e_0)-c(r(x)e_1)\right)^2\\
%&=&\left(r^*D+t\sum_{\alpha}D_{0,\alpha}'\right)^2 +
%c(e_1)^2(\partial_r+r(x))^2 \\
&=&\left(r^*D+t\sum_{\alpha}D_{0,\alpha}'\right)^2 +
(\partial_r+r(x))^2,  
\end{eqnarray*} 
and hence, 
%$$
%(\bD_t^{\rm mod})^2|_{\Gamma(W_0^{\pm}\otimes R)}=
%\left(r^*D+t\sum_{\alpha}D_{0,\alpha}'\right)^2\otimes 1 +
%(-\partial_r^2\otimes 1+r^2\otimes 1\pm 1\otimes c_{1,2}(e_0)c_{1,2}(e_1)). 
%$$ 
$$
\ker_{L^2}(\bD_t^{\rm mod})^2\cong 
\ker_{L^2}\left(D+t\sum_{\alpha}D_{0,\alpha}'\right)^2\otimes
\ker_{L^2}(\partial_r+r)^2
$$as $\Z/2$-graded vector spaces. 
Since 
the space $\ker_{L^2}(\partial_r+r)^2$ is a two-dimensional real vector space 
generated by $f_1(r)=e^{-\frac{r^2}{2}}$ and $f_2(r)=re^{-\frac{r^2}{2}}$ without $\Z/2$-grading, 
we complete the proof. 
\end{proof}

\subsection{Definition of $\bD_{t,a}^{\rm mod}$}
For $a\in \R$ and $t>0$ we define $\bD_{t,a}^{\rm mod}:\Gamma(\tilde W)\to \Gamma(\tilde W)$ by 
\begin{equation}\label{modelop2}
\bD^{\rm mod}_{t,a}=\left(r^*D+c(v_r)\partial_r+t\sum_{\alpha} D_{0,\alpha}'\right)c(e_0)-(r(x)-a)c(e_1).  
\end{equation}
Let $\tau_a:M_0\times \R\to M_0\times \R$, 
$\tau_a(x,r)=(x,r+a)$ be the translation. 
Then we have $\bD_{t,a}^{\rm mod}=\tau_{-a}^*\circ\bD_{t}^{\rm mod}\circ\tau_a^*$, and hence,  
$\ind(\bD_{t,a}^{\rm mod})=\ind(\bD_t^{\rm mod})$ for any $t>T^{\rm mod}$ and $a\in \R$. 
%, where $T$ is the constant appeared as in Proposition~\ref{indDta}. 
%
Let $\bD_{t,a}^{\pm}$  (resp. $\bD^{\rm mod, \pm}_t$ and $\bD_{t,a}^{\rm mod, \pm}$) be the restriction of 
$\bD_{t,a}$  (resp. $\bD^{\rm mod}_t$ and $\bD_{t,a}^{\rm mod}$) to $\Gamma(\tilde W^{\pm})$. 
%where the $\Z/2$-grading $\tilde W=\tilde W^+\oplus \tilde W^-$ is 
%given by $c(\epsilon)=\pm 1$. 

For a self-adjoint operator $P$ with discrete spectrum and a real number $\lambda\in \R$, 
let $N(\lambda,P)$ be the number of eigenvalues of $P$ less than or equal to $\lambda$. 
%As in the same way for the proof of Proposition~?? 
% and Theorem~\ref{} in \cite{Bravrman} 
%we can show the following. 

\begin{prop}\label{keyprop}
Let $\lambda_{\pm}$ be the smallest nonzero eigenvalue of $(\bD_t^{\rm mod, \pm})^2$. 
For any $\varepsilon>0$ there exists $A=A(\varepsilon,M,W,t)>0$ such that 
$$
N(\lambda_{\pm}-\varepsilon, (\bD_{t,a}^{\pm})^2)=\dim\ker_{L^2}(\bD_t^{\rm mod, \pm})^2 
$$for any $t>\max\{T,T^{\rm mod}\}$ and $a>A$. 
In particular we have 
$$\ind(\bD_{t,a})=\ind(\bD_t^{\rm mod}). $$
\end{prop}

%Proposition~\ref{keyprop} implies $\ind(\bD_{t,a})=\ind(\bD_t^{\rm mod})$. 
Together with Proposition~\ref{kermod=kerDt} the above Proposition~\ref{keyprop} implies 
the following theorem, and hence, the proof of 
Theorem~\ref{cobinvreal} is finished by Proposition~\ref{indDta=0}.

\begin{theorem}For any $t>\max\{T, T^{\rm mod}\}$ and $a>A$ we have 
$$
\ind(\bD_{t,a})=2\ind(M_0,V_0). 
$$
\end{theorem}

%Proposition~\ref{keyprop} can be shown by proving 
%$$
%N(\lambda_{\pm}-\varepsilon, (\bD_{t,a}^{\pm})^2)\leq\dim\ker(\bD_t^{\rm mod, \pm})^2  
%$$and 
%$$
%N(\lambda_{\pm}-\varepsilon, (\bD_{t,a}^{\pm})^2)\geq\dim\ker(\bD_t^{\rm mod, \%pm})^2. 
%$$
For each $a>0$ let $j,\bar j$, $J_a$ and $\bar J_a$ be the smooth functions as in 
\cite[Subsection~11.9]{Braverman}. 
Namely let $j,j:\R\to [0,1]$ be smooth functions such that 
$j^2+\bar j^2=1$, $j(r)=1$ for $r\geq 3$ and $j(r)=0$ for $r\leq 2$.  
We also define smooth functions $J_a, \bar J_a : M_0\times \R\to [0,1]$
 by 
$$
J_a(x):=j(a^{-1/2}r(x)), \quad  \bar J_a(x)=\bar j(a^{-1/2}r(x)),  
$$where $r:M_0\times\R\to M_0$ is the projection. 
We will denote by the same letters for smooth functions 
on $\tilde M$ defined by the formula 
$$
J_a(x):=j(a^{-1/2}p(x)), \quad  \bar J_a(x)=\bar j(a^{-1/2}p(x)), 
$$where $p:\tilde M\to \R$ is the smooth function defined in 
Subsection~\ref{The perturbation of Dirac-type operator}. 

The proof of Proposition~\ref{keyprop} is 
almost same as that for \cite[Proposition~11.6]{Braverman}. 
In \cite{Braverman} Braverman used IMS localization formula 
using $J_a $ and $\bar J_a$ due to \cite{Shubin}. 
We only note that we have to change \cite[Lemma~11.14 ]{Braverman} to the following lemma. 
\begin{lemma}
%Put $2\max\left\{\max\{|dj(t)|^2, |d\bar j(t)|^2|\} \ | \ t\in \R \right\}$. 
There exists a positive constant $K=K(M,W, t)>0$ such that 
$$
\|\left[J_a, [J_a, \bD^2_{t,a}]\right]\|, \|\left[\bar J_a, [\bar J_a, \bD^2_{t,a}]\right]\|\leq\frac{K}{a}
$$for any $a>0$.
\end{lemma}
\begin{proof}
The estimate follows from the fact that 
the principal symbol of $\bD_{t,a}$ has the form 
$$
c\circ({\rm id}+t\sum_{\gamma}\rho_{\gamma}^2\iota_{\gamma}^*)c(e_0), 
$$where $c$ is the Clifford multiplication on $\tilde W$ and 
$\iota_{\gamma}^*$ is the dual of the natural inclusion 
from $T\hat\SF_{\gamma}$ to $T\hat V_{\gamma}$. 
\end{proof}

\subsection{Comments on the proof of complex case}
\label{Comments on the complex case}
In this subsection we give comments on the proof of Theorem~\ref{cobinvref}, 
which can be proved by the almost same argument as that for 
Theorem~\ref{cobinvreal} or in \cite{Braverman, Bravermancobinv}. 
We explain how we modify the perturbation etc. 

Let $M$ be a Riemannian manifold with boundary $\partial M=M_0$. 
Let $\hat W\to M$ be the complex Clifford module bundle as in the set-up in 
Subsection~\ref{Complex Clifford module bundle case}. 
Namely $\hat W$ is a complex $Cl(TM\oplus \R e_L\oplus \R^{p,q}\oplus \R e_R \oplus \R e_R')$-module bundle, where the metric is defined by $e_L\cdot e_L=1$ and 
$e_R\cdot e_R=e_R'\cdot e_R'=-1$.  
Let $W={\rm Hom}_{Cl_{1,2}}(R_{\C}, \hat W)\to M$ be the associated $Cl(TM\oplus \R^{p,q})$-module bundle for the irreducible complex representation 
$R_{\C}$ of $Cl_{1,2}=Cl(\R e_L\oplus \R e_R\oplus \R e_R')$. 
Suppose that there exists an open subset $V$ of $M$ whose complement is compact and $V$ is equipped with an acyclic compatible system 
$\{V_{\gamma}, \SF_{\gamma}, \hat W, \hat D_{\gamma}\}_{\gamma\in \hat A}$. 
By attaching the cylinder $M_0\times [0,\infty)$ we have the 
associated non-compact manifold $\tilde M$, its open subset $\tilde V$ 
and the induced $G$-tangential acyclic compatible system 
$\{\tilde V_{\gamma}, \tilde\SF_{\gamma}, \tilde W, \tilde D_{\gamma}\}_{\gamma\in \hat A}$ on $\tilde V$. 
When we put $V_0:=M_0\cap V$ 
we have the induced $G$-tangential acyclic compatible system 
$\{V_{0,\alpha}, \SF_{0,\alpha}, W_0, D_{0,\alpha}\}_{\alpha\in A}$ 
over $V_0$, 
where $W_0=W|_{M_0}=W_0^+\oplus W_0^-$ is  the $\Z/2$-graded $Cl(TM_0)$-module 
whose grading is given by $c(v_r)=\pm\sqrt{-1}$. 
We first take and fix a Dirac-type operator $D$ on $\Gamma(W_0)$ and 
$\tilde D$ on $\Gamma(\tilde W)$ such that 
\begin{equation}
\tilde D|_{M_0\times (0,\infty)}=\sqrt{-1}\left(\pi^*D+c(v_r)\partial_r\right)\otimes c_{1,2}(e_L) 
\end{equation}
under the isomorphism $\tilde W|_{M_0\times (0,\infty)}\cong \pi^*W_0\otimes 
R_{\C}$. 
For $t>0$ and $a\in \R$ we consider the perturbations 
$$
D_t=\tilde D+t\sum_{\gamma\in \hat A}\tilde D_{\gamma}' \quad {\rm and} \quad 
\bD_{t,a}=D_t-c((p(x)-a)e_R)
$$
instead of the perturbation~(\ref{perturbation1}) and 
(\ref{perturbation2}), which acts on $\Gamma(\tilde W)$. 
As in the same way we consider the model operator on $\Gamma(\tilde W|_{M_0}\times \R)$ and 
its perturbation 
$$
\bD_{t}^{\rm mod}=\sqrt{-1}\left(r^*D+c(v_r)\partial_r+t\sum_{\alpha\in A_0}D_{0,\alpha}'\right)\otimes c_{1,2}(e_L)-1\otimes c_{1,2}((p(x))e_R),  
$$
$$
\bD_{t,a}^{\rm mod}=\sqrt{-1}\left(r^*D+c(v_r)\partial_r+t\sum_{\alpha\in A_0}D_{0,\alpha}'\right)\otimes c_{1,2}(e_L)-1\otimes c_{1,2}((p(x)-a)e_R).
%=\tau_{-a}^*\circ \bD_{t}^{\rm mod}\circ \tau_a^*. 
$$
Using these operators we can give a proof of Theorem~\ref{cobinvref} 
as for the same argument for the proof of Theorem~\ref{cobinvreal}. 
%%%%%%%%%%%%%%%%%%%%%%%%%%%%%%%%%%%%%%%%%
\section{Well-definedness of the local index}
\label{Well-definedness of the local index} 
Though the local index is invariant under the continuous deformation of the data, it is not clear that it does not depend on the choice of the open covering. 
We show that the local index does not depend on it for a suitable class of 
compatible systems. 

\begin{theorem}\label{well def of local ind}
Let $M$ be a Riemannian manifold. 
Let $W\to M$ be a $\Z/2$-graded $Cl(TM\oplus \R^{p,q})$-module bundle. 
Let $V_1$ and $V_2$ be two open subsets whose complements in $M$ are compact. 
Suppose that for $i=1,2$ there exist $G$-tangential acyclic compatible systems 
$\{V_{i,\alpha}, \SF_{i,\alpha}, W|_{V_i}, D_{i,\alpha}\}_{\alpha\in A_i}$. 
If the union 
$\{V_{1,\alpha}, V_{2,\beta}, \SF_{1,\alpha},\SF_{2,\beta}, W|_{V_1\cup V_2}, D_{1,\alpha}, D_{2,\beta}\}_{\alpha,\beta\in A_1\cup A_2}$ is also a $G$-tangential 
acyclic compatible system on $V_1\cup V_2$, then 
we have 
$\ind(M,V_1)=\ind(M,V_2)$. 
\end{theorem}
\begin{proof}
We put $\tilde M:=M\times [0,1]$ and $\tilde V:=V_1\times [0,2/3]\cup V_2\times [1/3,1]$. 
Then we have a cobordism 
$\{V_{1,\alpha}\times [0,2/3], V_{2,\beta}\times[1/3,1], \SF_{1,\alpha}\times [0,2/3],\SF_{2,\beta}\times[1/3,1], W|_{V_1\cup V_2}\times[0,1], D_{1,\alpha}, D_{2,\beta}\}_{\alpha,\beta\in A_1\cup A_2}$ between 
$\{V_{1,\alpha}, \SF_{1,\alpha}, W|_{V_1}, D_{1,\alpha}\}_{\alpha\in A_1}$ and 
$\{V_{2,\alpha}, \SF_{2,\alpha}, W|_{V_2}, D_{2,\alpha}\}_{\alpha\in A_2}$, 
and hence,  we have $\ind(M,V_1)=\ind(M,V_2)$ by Theorem~\ref{maintheorem}.  
\end{proof}

%%%%%%%%%%%%%%%%%%%%%%%%%%%%%%%%%%%%%
%\section{Applications}
\subsection{Application to the product of compatible fibrations}
 
For $i=1,2$ let $M_i$ be Riemannian manifolds and 
$W_i\to M_i$ $\Z/2$-graded $Cl(TM_i\oplus \R^{p,q})$-module bundles. 
Suppose that for each $i=1,2$ there exists a compact Lie group $G_i$ 
which acts on $M_i$ so that $W_i\to M_i$ is a $G_i$-equivariant 
$Cl(TM_i\oplus \R^{p,q})$-module bundle. 
Let $M_{i}'$ be the complement of the fixed point set, 
$M_i'=M_i\setminus M_i^{G_i}$.  
Suppose that there exists a $G_i$-tangential compatible fibration 
$\{M_{i,\alpha}',\SF_{i,\alpha}\}_{\alpha\in A_i}$ and a 
$G_i$-tangential compatible system $\{D_{i,\alpha}\}_{\alpha\in A_i}$ 
on it. 
Let $V_i$ be a $G_i$-invariant open subset of $M_i'$ such that 
the compatible system $\{D_{i,\alpha}\}_{\alpha\in A_i}$ is acyclic on $V_i$. 
We assume $M_i\setminus V_i$ is compact. 
Let $M$ be the product of $M_1$ and $M_2$ and 
$G$ the product of $G_1$ and $G_2$. 
As in \cite[Section~8]{Fujita-Furuta-Yoshida2} when we 
take small open neighborhoods $V_{i,\infty}$ of $M_i\setminus V_i$ for $i=1,2$ 
and put 
$V:=V_{1,\infty}\times V_2\cup V_1\times V_{2,\infty} 
\cup V_1\times V_2$, there exists a structure of $G$-tangential 
  acyclic compatible system on $V$. 
Under the above assumptions we can define three local indices
$\ind(M_1, V_1)$,  $\ind(M_2, V_2)$,  {\rm and}  
$\ind(M,V)$.  
%By the product formula of local indices (\cite[Theorem~]{Fujita-Furuta-Yoshida2}), we have 
%$$
%\ind(M,V)=\ind(M_1,V_1)\ind(M_2,V_2). 
%$$

On the other hand we can define a refinement of compatible fibration on $V$ 
using the decomposition
$$
V=M_{1,\infty}'\times V_2\cup M_1'\times V_2
\cup V_1\times M_{2,\infty}'\cup V_1\times M_2', 
$$where $M_{i,\infty}'$ is a small open neighborhood of $M_i\setminus M_i'=M_{i}^{G_i}$. 
Here we consider the structure of compatible fibration on $V$ as follows. 
\begin{itemize}
\item $M_{1,\infty}'\times V_2$ : product of the trivial foliation on $M_{1,\infty}'$ and $\{\SF_{2,\beta}|_{V_2}\}_{\beta\in A_2}$. 
\item $M_{1}'\times V_2$ : product of $\{\SF_{1,\alpha}|_{V_1}\}_{\alpha\in A_1}$ and $\{\SF_{2,\beta}|_{V_2}\}_{\beta\in A_2}$. 
\item $V_1\times M_{2,\infty}'$ : product of $\{\SF_{1,\alpha}|_{V_1}\}_{\alpha\in A_1}$ and the trivial foliation on $M_{2,\infty}'$. 
\item $V_1\times M_2'$ : product of $\{\SF_{1,\alpha}|_{V_1}\}_{\alpha\in A_1}$ and $\{\SF_{2,\beta}|_{V_2}\}_{\beta\in A_2}$. 
\end{itemize}
Note that even if the compatible system $\{D_{1,\alpha}|_{M_1'}\}_{\alpha\in A_1}$ 
is not acyclic, the product of it and 
$\{D_{2,\beta}|_{V_2}\}_{\beta\in A_2}$ is acyclic. 

In this way we have an another acyclic compatible system on $V$, 
which is also $G$-tangential and different from the product of 
the given compatible systems. 
We denote the open subset $V$ with this refined structure by $V^{\rm ref}$. 
We can define the local index $\ind(M,V^{\rm ref})$ 
for this refined structure. 

\begin{theorem}\label{ind=ind^ref}
Under the above setting we have 
$$
\ind(M,V^{\rm ref})=\ind(M,V)=\ind(M_1,V_1)\ind(M_2,V_2). 
$$
\end{theorem}
\begin{proof}
%Consider the product of $M$ and the closed interval $ [0,1]$. 
%Then the open covering of $V\times [0,1]$ defined by 
%$V \times [0,2/3] \cup V^{\rm ref} \times [1/3,1]$ induces an 
%%acyclic cobordism of two acyclic compatible systems on $V$ and $V^{\rm ref}$, 
%which is $G$-tangential. 
The first equality follows from Theorem~\ref{well def of local ind}. 
The second equality follows from 
the product formula of local indices (\cite[Theorem~8.8]{Fujita-Furuta-Yoshida2}). 
\end{proof}

\begin{remark}
Since we assume the global group action, 
when all the data are $G_i$-equivariant 
the equality in Theorem~\ref{ind=ind^ref} can be 
regarded as an equality between $G$-equivariant indices. 
\end{remark}

\begin{remark}
It is possible to give a formulation of 
the similar equality as in Theorem~\ref{ind=ind^ref} 
for a twisted product of $M_1$ and $M_2$ using a principal bundle over $M_1$. 
\end{remark}

\begin{example}
Let $M_1$ be the product of $S^1$ and the open interval $(-1,1)$. 
Consider the standard $S^1$-action on the first factor. 
It induces a structure of $S^1$-tangential compatible fibration on $M_1$. 
In this case we have $M'_1=M_1$. 
Let $W$ be a $\Z/2$-graded Clifford module bundle over $M_1$. 
Suppose that we have a family of Dirac-type operator 
$\{D_{S^1, r}:\Gamma(W|_{S^1\times \{r\}})\to \Gamma(W|_{S^1\times \{r\}})\}_{r\in(-1,1)}$. We assume that $\ker(D_{S^1,r})=\{0\}$ for all $r\neq 0$. 
Then $\{D_{S^1, r}\}_{r\neq 0}$ induce an acyclic compatible system on 
$V_1:=M_1\setminus S^1\times\{0\}$. 
Consider the product $M:=M_1\times M_1$. 
When we put 
$$
V_{01}:=(S^1\times (-1/4,1/4))\times V_1, \quad 
V_{10}:=V_1\times (S^1\times (-1/4,1/4)),  
$$
$$
V_{11}:=V_1\times V_1, \ {\rm and} \  
V:=V_{01}\cup V_{10}\cup V_{11},  
$$
it induces a structure of compatible system 
using the following torus actions.  
\begin{itemize}
\item $V_{01}$ : $\{e\}\times S^1$-action. 
\item $V_{10}$ : $S^1\times \{e\}$-action. 
\item $V_{11}$ : $S^1\times S^1$-action. 
\end{itemize}
On the other hand when we put 
$$
V_{01}'=M_1\times V_1, \ V_{10}'=V_1\times M_1, \ {\rm and} \  V^{\rm ref}:=V_{01}'\cup V_{10}'(=V), 	
$$it induces a structure of compatible system 
using the $S^1\times S^1$-actions on $V_{01}'$ and $V_{10}'$. 
See Figure~\ref{cylinderproduct}.

\begin{figure}[h]
\begin{center}
\input{cylinderproduct.tex}
\caption{$V$ and $V^{\rm ref}$}\label{cylinderproduct}
\end{center}
\end{figure}

Theorem~\ref{ind=ind^ref} guarantees that 
two local indices defined by using these two different structures coincide: 
$\ind(M,V)=\ind(M,V^{\rm ref})$. 

\end{example}

\noindent{\bf Acknowledgements.}
The author would like to thank 
Mikio Furuta and Takahiko Yoshida 
for stimulating discussions. 

\bibliographystyle{amsplain}
\bibliography{reference}

\end{document}